\documentclass[12pt]{article}
\usepackage[english]{babel}
\usepackage{amssymb,amsmath,amsthm}
\newcommand{\eps}{\varepsilon}
\newcommand{\D}{\mathrm{d}}
\newcommand{\OO}{\mathcal{O}}
\newtheorem{claim}{Claim}[section]
\newtheorem{thm}[claim]{Theorem}
\newtheorem{rems}[claim]{Remarks}

\title{Approximations of quantum-graph vertex couplings by singularly scaled rank-one operators}

\author{P. EXNER$^{1,2}$ and S.S. MANKO$^{1,3}$\footnote{On leave of absence from Pidstryhach Institute for Applied Problems of Mechanics and Mathematics, National Academy of Sciences of Ukraine, 3b Naukova str, 79060 Lviv, Ukraine}}
\date{\small $^1$Doppler Institute for Mathematical Physics and Applied Mathematics, \\ Czech Technical University in Prague, B\v{r}ehov\'{a} 7, 11519 Prague, Czechia \\ $^2$Nuclear Physics Institute ASCR, 25068 \v{R}e\v{z} near Prague, Czechia \\
$^3$Department of Physics, Faculty of Nuclear Science and Physical Engineering, \\ Czech Technical University in Prague, Pohrani\v{c}n\'{\i} 1288/1,  40501 D\v{e}\v{c}\'{\i}n, Czechia \\ \emph{e-mail: exner@ujf.cas.cz, stepan.manko@gmail.com}}

\begin{document}

\maketitle

\noindent \textbf{Abstract.} We investigate approximations of the vertex coupling on a star-shaped graph by families of operators with singularly scaled rank-one interactions. We find a family of vertex couplings, generalizing the $\delta'$-interaction on the line, and show that with a suitable choice of the parameters they can be approximated in this way in the norm-resolvent sense. We also analyze spectral properties of the involved operators and demonstrate convergence of the corresponding on-shell scattering matrices.

\medskip


\medskip


\section{Introduction}

Quantum graphs are a versatile model of many physical systems; we refer to the recent monograph \cite{BK} for an extensive bibliography. One of the central items of this theory are vertex coupling conditions used to match wave functions supported by graph edges. From general principles they have to be chosen to make the graph Hamiltonian self-adjoint. This is a simple task but the result leaves a lot a freedom through parameters entering those conditions the values of which have to be fixed. The background of such a choice is the question about the physical meaning of the coupling, important without any doubt; one has to keep mind that different vertex couplings give rise to different quantum dynamics on the graph.

A natural approach to the problem is to analyze various approximations of those couplings, either on the graph itself or using a tubular network shrinking to the graph `skeleton'. Even in the latter case, however, one considers approximations on the graph as an intermediate step -- cf. \cite{EP} and references therein. The task simplifies due to the fact that the vertex couplings, the physically interesting ones at least, are of a local nature; it is thus sufficient to solve the problem for a star-type graph with $n$ edges meeting in a single vertex. One usually begins with the most simple coupling, often called Kirchhoff, and investigates families of scaled interaction supported in the vicinity of the junction. The simplest example are potentials scaled with their mean preserved which give rise to one-parameter family of the so-called \mbox{$\delta$-couplings} \cite{Ex1}. This is, however, only a small subclass of the couplings allowed by the self-adjointness requirement and using potentials scaled in a more singular way many other conditions can be obtained -- without going into details we refer to \cite{EM} and the bibliography therein.

These scaled-potential approximations, however, do not yield all the admissible couplings, in particular, one cannot obtain in this way strongly singular matching conditions such as the so-called $\delta'$-coupling and its modifications, which are of interest mainly because their properties contrast in a sense to those of more regular couplings, for instance, a $\delta'$ junction is opaque at high energies.

Note that various results are known in the simplest nontrivial case $n=2$ where the coupling is nothing else that a generalized point interaction on the line \cite{AGHH}. In particular, an approximation of the $\delta'$-interaction on the line with the help of scaled rank-one operators was proposed longtime ago by \v{S}eba \cite{Se}. The aim of the present paper is to propose and analyze a similar approximation of a class of singular vertex couplings, given by relations (\ref{match}) below, by nonlocal potentials for a general star-shaped graph. Our main results are demonstration of the norm-resolvent convergence of such an approximation, Theorem~\ref{thm:LimOpSpectr} below, and convergence of the corresponding on-shell scattering matrices, Theorem~\ref{thm:ScatMatConv}. Let us add that as in the case $n=2$ the constructed approximation is non-generic, and furthermore, it represents a new generalization of the $\delta'$-interaction on the line, different from the two known ones \cite{Ex2}; in contrast to those it leans on the permutation asymmetry of the approximating operators.

\section{Preliminaries}

To begin with, we recall a few basic notions concerning metric graphs. In what follows, we focus on noncompact star-shaped graphs $\Gamma$ consisting of $n\in\mathbb{N}$ semi-infinite edges $\gamma_1,\dots,\gamma_n$ connected at a single vertex. A map $\psi: \Gamma\to\mathbb{C}$ is said to be a function on the graph and its restriction to the edge $\gamma_i$ will be denoted by $\psi_i$. Each edge $\gamma_i$ has a natural parametrization $x_i$ given by the arc length of the curve representing the edge, hence without loss of generality we may identify each $\gamma_i$ with the halfline $[0,\infty)$.
A differentiation is always related to this natural length parameter. We denote by $\psi'_i(0)$ the limit value of the derivative at the graph vertex taken conventionally in the outward direction, i.e. away from the vertex. The integral $\int_\Gamma \psi\,\D x$ of $\psi$ over $\Gamma$ is the sum of integrals over all edges $\sum_{i=1}^n\int_0^\infty \psi_i\,\D x_i$, the measure being the natural Lebesgue measure.

As usual, $L^2(\Gamma)$ denotes the Hilbert space of (equivalence classes of) such functions with the scalar product $(\psi_1,\psi_2) =\int_\Gamma\psi_1 \bar\psi_2\,\D x$. Next we introduce the Sobolev space $H^2(\Gamma)$ as the Banach space with the norm $\|\psi\|_{H^2(\Gamma)}= (\|\psi\|_{L^2(\Gamma)}+ \|\psi''\|_{L^2(\Gamma)})^{1/2}$, observing that neither the functions belonging to $H^2(\Gamma)$ nor their derivatives should be continuous at the graph vertex. Finally, we say that a function $\psi$ satisfies the Kirchhoff conditions at the graph vertex and write $\psi\in K(\Gamma)$ if $\psi$ is continuous at this vertex and satisfies the condition $\sum_{i=1}^n\psi'_i(0)=0$.

Given a star graph $\Gamma$ described above, we introduce the following family of Schr\"{o}dinger operators on $L^2(\Gamma)$ labeled by the parameter $\eps\in(0,1]$,
\begin{equation} \label{approxop}
-\Delta^\eps:=-\frac{\D^2}{\D x^2}+
\frac{\lambda(\eps)}{\eps^3}V_\eps(x)
\big\langle\cdot,V_\eps\big\rangle_\Gamma\,,
\qquad \mathfrak{D}(-\Delta^\eps)=H^2(\Gamma)
\cap K(\Gamma)\,,
\end{equation}
with $V_\eps(\cdot):=V(\frac\cdot\eps)$, where the real-valued function $V$ belongs to the class $L^1_\mathrm{loc}(\Gamma)$ and has a compact support and zero mean, i.e., $\int_\Gamma V\,\D x=0$. Without loss of generality we may (and shall) suppose that the support of $V$ is contained in the unit ball centered at the graph vertex. With respect to the edge indices, $V$ may be regarded as an $n\times n$ matrix function on $[0,\infty)$; we stress that it need not be diagonal. In a similar vein the differential part of $-\Delta^\eps$ is a shorthand for the operator which acts as the negative second derivative on each edge $\gamma_i$. The function $\lambda(\cdot)$ in the above expression is supposed to be real-valued for real $\eps$ and holomorphic in the vicinity of the origin. In addition, it satisfies the condition
\begin{equation} \label{scaling}
\lambda(\eps)=\lambda_0+\eps\lambda_1+\OO(\eps^2)\,,
\qquad \eps\to0\,,
\end{equation}
where $\lambda_0$ and $\lambda_1$ are nonzero real numbers. Our main goal in this paper is to investigate convergence of the operators $-\Delta^\eps$ as $\eps\to 0$ in the norm-resolvent topology.

To describe the outcome of the limiting process we need the following quantities:
\begin{equation} \label{form:theta}
\vartheta_i:=\int_0^\infty x_i V(x_i)\,\D x_i\,,\quad i=1,2,\dots,n\,,
\end{equation}
and
\begin{equation}\label{form:A}
A:=-
\sum_{i=1}^n
\int_0^\infty\int_0^\infty
\min\{x_i,y_i\}
\,
V(x_i)V(y_i)
\,\D x_i\D y_i
\,.
\end{equation}
Using them, we define the {\it limit operator} $-\Delta_\beta$ as the one acting as
\[
-\Delta_\beta\psi:=-\psi''
\]
on functions $\psi\in{H}^2(\Gamma)$ that obey the matching conditions
\begin{equation} \label{match}
\frac{\psi_i(0)-\psi_j(0)}
{\vartheta_i-\vartheta_j}=
\beta
\sum_{\ell=1}^n
\vartheta_\ell\, \psi'_\ell(0)\,,
\quad 1\leq i<j\leq n\,,\qquad
\sum_{\ell=1}^n
\psi'_\ell(0)=0\,,
\end{equation}
where we adopt the convention that $\psi_i(0)=\psi_j(0)$ holds if $\vartheta_i=\vartheta_j$. Here
\begin{equation}\label{beta}
\beta=
\begin{cases}
\frac1{\lambda_1A^2}&\textnormal{if $\lambda_0A=1$}\,,\\
\phantom{-}0&\textnormal{otherwise}.
\end{cases}
\end{equation}

\begin{rems} {\rm
\begin{itemize} \setlength{\itemsep}{-3pt}
\item [(i)] In the particular case $n=2$ one gets from \eqref{match} the well-known one-dimensional $\delta'$-interaction with the coupling parameter $\beta$.
\item [(ii)] The boundary conditions that determine the domain of the limit operator $-\Delta_\beta$ can alternatively be written in the conventional form \cite{KS} as
\begin{equation}\label{eq:Bound_Cond}
\mathcal{A}\Psi(0)+\mathcal{B}\Psi'(0)=0\,,
\end{equation}
where
\[
\mathcal{A}=
\begin{pmatrix}
0   &   0   &   0   &   \dots   &   0
\\
\frac{1}{\vartheta_1-\vartheta_2}   &   \frac{1}{\vartheta_2-\vartheta_1}  &   0   &  \dots   &   0
\\
\frac{1}{\vartheta_1-\vartheta_3}   &   0   &   \frac{1}{\vartheta_3-\vartheta_1}  & \dots   &   0
\\
\vdots   &   \vdots   &   \vdots   &   \ddots  &   \vdots
\\
\frac{1}{\vartheta_1-\vartheta_n}   &   0   &   0  &   \dots   &   \frac{1}{\vartheta_n-\vartheta_1}
\end{pmatrix}
\,,
\quad
\mathcal{B}=
-\beta
\begin{pmatrix}
\frac1\beta   &   \frac1\beta   &   \dots   &   \frac1\beta
\\
 \vartheta_1&
\vartheta_2&
\dots&
\vartheta_n
\\
\vartheta_1&
\vartheta_2&
\dots&
\vartheta_n
\\
\vdots   &   \vdots   &   \ddots   &   \vdots
\\
\vartheta_1&
\vartheta_2&
\dots&
\vartheta_n
\end{pmatrix}\,.
\]
It is useful to adopt the convention that all entries of the first row of $\mathcal{B}$ are $-1$ even if $\beta=0$.
As it is well known,  to make a graph Hamiltonian self-adjoint, a graph vertex at which $n$ edges meet should be characterized by boundary conditions \eqref{eq:Bound_Cond} in which the $n\times n$ matrices $\mathcal{A}$, $\mathcal{B}$ are such that the product $\mathcal{AB}^*$ is self-adjoint, while the $2n\times n$ matrix $(\mathcal{A}|\mathcal{B})$ has rank $n$; it is straightforward to check that this is the case here.
\item [(iii)] If $\vartheta_1=\vartheta_j$ holds for some $j=2,\dots,n$,v then the $j$th row of the above matrix $(\mathcal{A}|\mathcal{B})$ should be replaced by the vector $(1,0,\dots,0,-1,0,\dots,0)$ of the length $2n$ with $-1$ at the $j$th place. In what follows, we assume without loss of generality that all the $\vartheta_i$ are different.
\end{itemize}
}\end{rems}

\noindent Having stated the problem, let us review briefly the following contents of the paper. In the next section, we describe properties of the limit operator $-\Delta_\beta$, in particular, we characterize its spectrum as well as its resolvent, and describe the corresponding scattering matrix -- cf.~Theorems~\ref{prop:Green_Funct}--\ref{thm:LimOp_ScatMat}. In Section~\ref{s:convergence} we first discuss the structure of the resolvent of the approximation operators \eqref{approxop} with the aim prove our first main result, namely its closeness to that of the limit operator, stated in Theorem~\ref{thm:ResolventConvergence}, together with its spectral consequences, Theorem~\ref{thm:Approx_Op_Spec}. Finally, Section~\ref{s:scattering} contains the other main results of the paper, Theorem~\ref{thm:ScatMatConv}, which says that the on-shell scattering matrices of the operators \eqref{approxop} approximate the corresponding on-shell S-matrix of the limit one.

\section{Limit operator} \label{s:limit}

After the above preliminaries, let us turn to basic properties of $-\Delta_\beta$. Starting from the explicit expression of the resolvent $(-\Delta_\beta-k^2)^{-1}$ we are going to describe the structure of the spectrum and scattering amplitudes of the limit operator.

\begin{thm}\label{prop:Green_Funct}
The resolvent $(-\Delta_\beta-k^2)^{-1}$ of the limit operator is an integral operator on $L^2(\Gamma)$ with the kernel
\begin{equation}\label{eq:Green_Function}
\Xi_k(x_i,y_j)=G_k(x_i,y_j)+\Lambda_{ij}(k^2)
\exp(\mathrm{i}k(x_i+y_j))\,,
\quad i,j=1,2,\dots,n\,,
\end{equation}
with $k^2\in\rho(-\Delta_\beta)$, $\Im k>0$, and with $\Lambda$ of the form
\begin{equation*}
\Lambda_{ij}(k^2)=
\frac{\beta\Pi_{ij}}{1+\mathrm{i}k\beta B}\,,
\end{equation*}
where
\begin{equation}\label{form:B}
B:=
\frac1n
\Bigg(
\sum_{i=1}^n
\vartheta_i
\Bigg)^2
-
\sum_{i=1}^n
\vartheta_i^2\,,
\quad
\Pi_{ij}:=
\Bigg(\frac1n\sum_{\ell=1}^n\vartheta_\ell-\vartheta_i\Bigg)
\Bigg(\frac1n\sum_{\ell=1}^n\vartheta_\ell-\vartheta_j\Bigg)
\,.
\end{equation}
In the relation \eqref{eq:Green_Function},
\begin{equation}\label{free:Green_Function}
G_k(x_i,y_j)=
\frac{\mathrm{i}}{2k}
\bigg[
\delta_{ij}\exp(\mathrm{i}k|x_i-y_j|)
+\bigg(\frac2n-\delta_{ij}\bigg)
\exp(\mathrm{i}k(x_i+y_j))
\bigg]
\end{equation}
is the integral kernel of the resolvent of the free Hamiltonian $-\Delta_0$.
\end{thm}
\begin{proof}
According to Krein's formula the sought Green's function is given by \eqref{eq:Green_Function} with the matrix $\Lambda$ to be found.
Suppose that $\psi$ solves the equation $(-\Delta_\beta-k^2)\psi=\phi$, then $\psi=(-\Delta_\beta-k^2)^{-1}\phi$ and relation \eqref{eq:Green_Function} allows us to write the function $\psi$ explicitly,
\[
\psi(x_i)=\sum_{j=1}^n\int_0^\infty\big(G_k(x_i,y_j)
+\Lambda_{ij}(k^2)\exp(\mathrm{i}k(x_i+y_j))\big)\phi(y_j)\,
\D y_j.
\]
Since the resolvent maps the whole space $L^2(\Gamma)$ onto the domain $\mathfrak{D}(-\Delta_\beta)$ of our operator, the function $\psi$ has to satisfy the boundary conditions \eqref{eq:Bound_Cond} at the vertex.
Using the explicit form of $G_k$, we find that
\begin{align*}
\psi\big|_{x_i=0}
&=\sum_{j=1}^n\bigg(\frac{\mathrm{i}}{kn}+\Lambda_{ij}(k^2)\bigg)
\int_0^\infty
\phi(y_j)\exp(\mathrm{i}ky_j)\,\D y_j\,,
\\
\psi'\big|_{x_i=0}
&=
\sum_{j=1}^n\bigg(\mathrm{i}k\Lambda_{ij}(k^2)+
\delta_{ij}-\frac{1}{n}\bigg)
\int_0^\infty
\phi(y_j)\exp(\mathrm{i}ky_j)\,\D y_j\,.
\end{align*}
Substituting the above relations into \eqref{eq:Bound_Cond}, we get a system of equations,
\begin{align*}
\sum_{j=1}^n
\sum_{i=1}^n
\int_0^\infty
\bigg[
\mathcal{A}_{\ell i}
\bigg(
\frac{\mathrm{i}}{kn}&+\Lambda_{ij}(k^2)
\bigg)
\\&+
\mathcal{B}_{\ell i}
\bigg(\mathrm{i}k\Lambda_{ij}(k^2)
+\delta_{ij}-\frac{1}{n}
\bigg)
\bigg]
\phi(y_j)\exp(\mathrm{i}ky_j)\,\D y_j=0
\end{align*}
for $\ell=1,2,\dots,n$. Next we require that the left-hand side vanishes for any $\phi$, which yields the condition
$
\mathcal{A}\tilde\Lambda+ \mathrm{i}k\mathcal{B}\tilde\Lambda+\mathcal{B}=0
$,
where $\tilde\Lambda_{ij}(k^2):=\Lambda_{ij}(k^2)+\frac{\mathrm{i}}{kn}$. This leads in a straightforward way to the following representation of the matrix $\tilde\Lambda(k^2)$,
\[
\tilde\Lambda(k^2)=-(\mathcal{A}+\mathrm{i}k\mathcal{B})^{-1}
\mathcal{B}.
\]
Next, we apply the Gauss elimination method to get the chain of equivalences
\[
\big(
    -(\mathcal{A}+\mathrm{i}k\mathcal{B})
    \big|
    \mathcal{B}
\big)
    \sim
    \dots
    \sim
\big(
    I
    \big|
    \underbrace{-(\mathcal{A}+\mathrm{i}k\mathcal{B})^{-1}
    \mathcal{B}}\limits_{\tilde\Lambda(k^2)}
\big);
\]
then by equivalent-row manipulations we pass to the matrix $(\mathcal{C}|\mathcal{D} )$, where
\[
\mathcal{C}=
\mathrm{i}k
\begin{pmatrix}
\prod\limits_{\ell=1}^nc_\ell    &0  &\dots&0&0
\\
0&\prod\limits_{\ell=2}^nc_\ell&\dots&0&0
\\
\vdots &\vdots &\ddots &\vdots &\vdots
\\
0 &0  &\dots &c_{n-1}c_n &0
\\
0 &0 &\dots & 0& c_n
\end{pmatrix},
\]
and
\[
\mathcal{D}=
\begin{pmatrix}
d_{11}
\prod\limits_{\ell=1}^{n-1}c_\ell
&
d_{12}
\prod\limits_{\ell=1}^{n-1}c_\ell
&\dots
&d_{1n-1}
\prod\limits_{\ell=1}^{n-1}c_\ell
&d_{1n}
\prod\limits_{\ell=1}^{n-1}c_\ell
\\
d_{21}
\prod\limits_{\ell=2}^{n-1}c_\ell
&
d_{22}
\prod\limits_{\ell=2}^{n-1}c_\ell
&\dots
&d_{2n-1}
\prod\limits_{\ell=2}^{n-1}c_\ell
&d_{2n}
\prod\limits_{\ell=2}^{n-1}c_\ell
\\
\vdots  &\vdots &\ddots &\vdots &\vdots
\\
d_{n-11}
c_{n-1}
&
d_{n-12}
c_{n-1}
&\dots
&d_{n-1n-1}
c_{n-1}
&d_{n-1n}
c_{n-1}
\\
d_{n1}
&
d_{n2}
&\dots
&
d_{nn-1}
&
d_{nn}
\end{pmatrix}
\]
with
$c_j:=j+\mathrm{i}k\beta \big(( \sum_{\ell=1}^j\vartheta_\ell)^2-j\sum_{\ell=1}^j\vartheta_\ell^2\big)$
and
$d_{ij}:=-1+\mathrm{i}k\beta(n\Pi_{ij}-B)$.
Consequently, we can divide each row of $(\mathcal{C}|\mathcal{D})$ by the corresponding diagonal element of $\mathcal{C}$. This yields the relation $(I|\tilde\Lambda)$, where the entries of $\tilde\Lambda$ are given by the formula $\tilde\Lambda_{ij}= \frac{d_{ij}}{\mathrm{i}kc_n}$. Thus
\begin{equation}\label{lambdaIJ}
\Lambda_{ij}=
\tilde\Lambda_{ij}+\frac1{\mathrm{i}kn}
=
\frac{1+\mathrm{i}k\beta B+d_{ij}}
{\mathrm{i}kn(1+\mathrm{i}k\beta B)}
=\frac{\beta\Pi_{ij}}{1+\mathrm{i}k\beta B},
\end{equation}
which therefore completes the proof of the theorem.
\end{proof}

\begin{thm}\label{thm:LimOpSpectr}
The essential spectrum of $-\Delta_\beta$ is purely absolutely continuous and covers the nonnegative real axis, while the singularly continuous spectrum is empty,
\[
    \sigma_\mathrm{ess}(-\Delta_\beta)=
    \sigma_\mathrm{ac}(-\Delta_\beta)=[0,\infty)\,,
    \quad
    \sigma_\mathrm{sc}(-\Delta_\beta)=\emptyset\,.
\]
If $\beta<0$, the operator $-\Delta_\beta$ has precisely one negative eigenvalue, namely its point spectrum $\sigma_\mathrm{p}(-\Delta_\beta)$ is
\[
\sigma_\mathrm{p}(-\Delta_\beta)=
\bigg\{-\frac{1}{\beta^2B^2}\bigg\}.
\]
If $\beta>0$, the limit operator has no eigenvalues,
\[
\sigma_\mathrm{p}(-\Delta_\beta)=\emptyset\,,\quad \beta\notin(-\infty,0)\,.
\]
\end{thm}
\begin{proof}
Since $(-\Delta_\beta-k^2)^{-1} -(-\Delta_0-k^2)^{-1}$, $k^2\in\rho(-\Delta_\beta)$, is of finite rank in view of \eqref{eq:Green_Function},
Weyl's essential spectrum theorem \cite[Theorem~XIII.14]{RS} implies that the essential spectrum of $-\Delta_\beta$ is not affected by the perturbation, i.e. $\sigma_\mathrm{ess}(-\Delta_\beta) =\sigma_\mathrm{ess}(-\Delta_0)=[0,\infty)$.
Using \eqref{eq:Green_Function} in combination with Theorem XIII.20 of \cite{RS}, one can check easily the absence of $\sigma_\mathrm{sc}(-\Delta_\beta)$. The structure of the point spectrum of $-\Delta_\beta$ for negative $\beta$ and the absence of negative eigenvalues for nonnegative one follow from the explicit meromorphic structure of the resolvent \eqref{eq:Green_Function}. On the other hand we note that the pole in the right-hand side of \eqref{eq:Green_Function} for $\beta>0$ corresponds to a resonance (antibound state). Finally, a short computation shows that the equation $-\Delta_\beta\psi=k^2\psi$ has no square integrable solutions for nonnegative $k$, which, in turn, leads to the absence of nonnegative eigenvalues for all real $\beta$.
\end{proof}

\begin{thm}\label{thm:LimOp_ScatMat}
For any momentum $k>0$ the on-shell scattering matrix $\mathcal{S}(k)$ for the pair $(-\Delta_\beta,-\Delta_0)$ takes the form
\[
\mathcal{S}_{ij}(k)
=\frac2n-\delta_{ij}-
\frac{2\mathrm{i}k\beta\Pi_{ij}}{1+\mathrm{i}k\beta B}
\]
with $\Pi_{ij}$ and $B$ given by relations \eqref{form:B}.
\end{thm}
\begin{proof}
The scattering matrix can easily be obtained by substituting the scattering solution $\psi(x_i)=\delta_{ij}\exp(-\mathrm{i}kx_j) +\mathcal{S}_{ij} \exp(\mathrm{i}kx_j)$ into the matching conditions \eqref{eq:Bound_Cond},
\[
\mathcal{S}(k)=-(\mathcal{A}+\mathrm{i}k\mathcal{B})^{-1}
(\mathcal{A}-\mathrm{i}k\mathcal{B})\,.
\]
Reasoning in a way similar to the proof of Theorem~\ref{prop:Green_Funct} we get the chain of equivalences
\[
\big(
    -(\mathcal{A}+\mathrm{i}k\mathcal{B})
    \big|
    (\mathcal{A}-\mathrm{i}k\mathcal{B})
\big)
    \sim
    \dots
    \sim
\big(
    I
    \big|
    \underbrace{-(\mathcal{A}+\mathrm{i}k\mathcal{B})^{-1}
    (\mathcal{A}-\mathrm{i}k\mathcal{B})}\limits_{\mathcal{S}(k)}
\big).
\]
Let the numbers $c_j$ and the  matrix $\Pi$ be the same as in the mentioned proof and set
\[
e_{ij}:=\frac2n(c_n
-\mathrm{i}kn^2\beta\Pi_{ij});
\]
using again row manipulations we to pass to the matrix $(\mathcal{C}|\mathcal{E} )$, where
\[
\mathcal{E}=
\mathrm{i}k
\begin{pmatrix}
(e_{11}-c_n)
\prod\limits_{\ell=1}^{n-1}c_\ell
&
e_{12}
\prod\limits_{\ell=1}^{n-1}c_\ell
&\dots
&e_{1n-1}
\prod\limits_{\ell=1}^{n-1}c_\ell
&e_{1n}
\prod\limits_{\ell=1}^{n-1}c_\ell
\\
e_{21}
\prod\limits_{\ell=2}^{n-1}c_\ell
&
(e_{22}-c_n)
\prod\limits_{\ell=2}^{n-1}c_\ell
&\dots
&
e_{2n-1}
\prod\limits_{\ell=2}^{n-1}c_\ell
&
e_{2n}
\prod\limits_{\ell=2}^{n-1}c_\ell
\\
\vdots  &\vdots &\ddots &\vdots &\vdots
\\
e_{n-11}
c_{n-1}
&
e_{n-12}
c_{n-1}
&\dots
&
(e_{n-1n-1}-c_n)
c_{n-1}
&
e_{n-1n}
c_{n-1}
\\
e_{n1}
&e_{n2}
&\dots
&e_{nn-1}
&e_{nn}-
c_n
\end{pmatrix}.
\]
Finally, we divide each row of $(\mathcal{C}|\mathcal{E})$ by the corresponding diagonal entry of $\mathcal{C}$, obtaining thus $(I|\mathcal{S})$,
where $\mathcal{S}$ is the sought scattering matrix and its entries are given by $\mathcal{S}_{ij}= \frac{e_{ij}}{c_n}-\delta_{ij}$.
To complete the proof it is sufficient to use the explicit formul{\ae} for $e_{ij}$ and $c_n$.
\end{proof}

\section{Convergence of the resolvents and spectra} \label{s:convergence}

This section is devoted to proof of the fact that the operator family $(-\Delta^\eps-k^2)^{-1}$ approximates $(-\Delta_\beta-k^2)^{-1}$ in the uniform operator topology. We will do that by demonstrating that the kernel of $(-\Delta^\eps-k^2)^{-1}$ approaches the kernel $\Xi_k$ of $(-\Delta_\beta-k^2)^{-1}$  in $L^2(\Gamma)$ given by Theorem~\ref{prop:Green_Funct}, which, in turn, makes it possible to verify the convergence
of the corresponding operators in the Hilbert-Schmidt norm, and thus, a fortiori, in the uniform norm. To this aim, we are going to construct
the resolvent $(-\Delta^\eps-k^2)^{-1}$ explicitly; this resolvent allows us to determine the structure of the spectrum of the perturbed operator, and we show that the spectrum of  $-\Delta^\eps$ is close to that of $-\Delta_\beta$ for small $\eps$. Our first main result reads

\begin{thm}\label{thm:ResolventConvergence}
As $\eps\to0$,
the family of Hamiltonians $-\Delta^\eps$
converges to $-\Delta_\beta$ in the norm-resolvent sense.
\end{thm}
\begin{proof}
To compare the resolvents of $\Delta^\eps$ and $\Delta_\beta$, fix $k:=\mathrm{i}\varkappa$  belonging to the resolvent sets of both operators; as we shall see later from the explicit structure of the resolvents this can be achieved, e.g., by choosing $\varkappa>0$ large enough.

We observe that the resolvent $(-\Delta^\eps+\varkappa^2)^{-1}$ is an integral operator in $L^2(\Gamma)$ which has the kernel of the following form,
\begin{align}
\begin{aligned}
\label{Resolv_Perturbed}
(-\Delta^\eps+\varkappa^2)^{-1}(x_i,y_j)
&=
G_{\mathrm{i}\varkappa}(x_i,y_j)
\\
&\phantom{=}\,-\zeta_\eps
\big((-\Delta_0+\varkappa^2)^{-1}
V_\eps\big)(x_i)
\big((-\Delta_0+\varkappa^2)^{-1}V_\eps\big)(y_j)\,,
\end{aligned}
\end{align}
with $G_{\mathrm{i}\varkappa}$ from \eqref{free:Green_Function} being the Green function of the free Hamiltonian $-\Delta_0$, and with the constant $\zeta_\eps$ of the form
\[
\zeta_\eps:=
\bigg(\frac{\eps^3}{\lambda(\eps)}+
\langle (-\Delta_0+ \varkappa^2)^{-1}V_\eps,V_\eps\rangle_\Gamma
\bigg)^{-1}.
\]
This expression is obtained in the same way as in the particular case $n=2$, i.e. for point interactions on the line -- see, e.g., \cite{AK}.

We start with the asymptotic behavior of the expression $\zeta_\eps$ as $\eps\to0$. Using the Taylor expansion of $\exp(-\eps\varkappa(x_i+y_i))$ together with the fact that $V$ has a compact support and zero mean, one derives the formula
\begin{align*}
\langle (-\Delta_0+\varkappa^2)^{-1}V_\eps,V_\eps
\rangle_\Gamma
&=
\frac{\eps^2}{2\varkappa}
\Bigg[
\sum_{i=1}^n
\int_0^\infty \int_0^\infty
V(x_i)V(y_i)
\exp(-\eps\varkappa|x_i-y_i|)
\,\D x_i\D y_i
\\
&\phantom{=}\,\,
-
\sum_{i=1}^n
\int_0^\infty \int_0^\infty
V(x_i)V(y_i)
\exp(-\eps\varkappa(x_i+y_i))
\,\D x_i\D y_i
\\
&\phantom{=}\,\,
+\frac2n
\sum_{i=1}^n
\sum_{j=1}^n
\int_0^\infty \int_0^\infty
V(x_i)V(y_j)
\exp(-\eps\varkappa(x_i+y_j))
\,\D x_i\D y_i
\Bigg]
\\
&=
-A\eps^3 +\varkappa B\eps^4+\mathcal{O}(\eps^5)\quad \mathrm{as} \;\;\eps\to0\,.
\end{align*}
Recall that the constants $A$ and $B$ are defined via formul{\ae} \eqref{form:A} and \eqref{form:B}, respectively. We thus conclude that
\[
\zeta_\eps=
\frac{1}{\eps^3\big(\frac1{\lambda_0}-A+\eps(\varkappa B-\frac{\lambda_1}{\lambda_0^2})\big)
}
+\mathcal{O}\bigg(\frac1{\eps^2}\bigg)\,,
\quad\eps\to0\,,
\]
and finally, since $\beta=\frac{\lambda_0^2}{\lambda_1}$ when $\lambda_0=\frac1A$, that
\begin{equation}\label{zeta_epsilon}
\zeta_\eps=
\frac{-\beta}{\eps^4(1-\varkappa\beta B)}
+\OO\bigg(\frac1{\eps^3}\bigg)\,,\quad
\eps\to0\,.
\end{equation}

In the next step we have to discuss the asymptotical behavior of the functions $((-\Delta_0+\varkappa^2)^{-1}V_\eps)(x_i) ((-\Delta_0+\varkappa^2)^{-1} V_\eps)(y_j)$. In a similar manner as above we find that
\begin{align*}
\big((-\Delta_0+\varkappa^2)^{-1}V_\eps\big)(x_i)
&=
\sum_{j=1}^n
\int_0^\infty
G_{\mathrm{i}\varkappa}(x_i,y_j)V_\eps(y_j)
\,\D y_j
\\
&=
-\eps^2\,
\exp(-\varkappa x_i)
\Bigg[
\sum_{j=1}^n
\bigg(\frac1n-\delta_{ij}\bigg)\vartheta_j
+\mathcal{O}(\eps)
\Bigg]\,,\quad \eps\to0\,,
\end{align*}
which, in turn, gives the desired relation,
\begin{multline}
\label{eq:VV}
\big(
    (-\Delta_0+\varkappa^2)^{-1}
    V_\eps
\big)(x_i)
\big(
    (-\Delta_0+\varkappa^2)^{-1}V_\eps
\big)(y_j)
\\
=\eps^4\exp(-\varkappa(x_i+y_j))
\big(\Pi_{ij}+\mathcal{O}(\eps)\big)\,,\;\;\;\eps\to0\,.
\end{multline}
Combining \eqref{Resolv_Perturbed}--\eqref{eq:VV} we find that
\begin{align*}
(-\Delta^\eps+\varkappa^2)^{-1}(x_i,y_j)
&=
G_{\mathrm{i}\varkappa}(x_i,y_j)
+
\exp(-\varkappa(x_i+y_j))
\bigg(\frac{\beta\Pi_{ij}}
{1-\varkappa\beta B}
+\mathcal{O}(\eps)\bigg),\\
&=
G_{\mathrm{i}\varkappa}(x_i,y_j)
+
\exp(-\varkappa(x_i+y_j))
\big(\Lambda_{ij}(-\varkappa^2)
+\mathcal{O}(\eps)\big)
\end{align*}
holds as $\eps\to0$. This allows us to conclude that the kernel $(-\Delta^\eps+\varkappa^2)^{-1}(x_i,y_j)$ of the approximating operator resolvent converges as $\eps\to0$ to the kernel $\Xi_{\mathrm{i}\varkappa}(x_i,y_j)$ pointwise. We further observe that the function $(-\Delta^\eps +\varkappa^2)^{-1}(x_i,y_j)$ decays exponentially, hence the integral expressing its norm converges and the dominated convergence theorem implies
that this function tends to the kernel $\Xi_{\mathrm{i}\varkappa}(x_i,y_j)$ in $L^2(\Gamma\times\Gamma)$. From this the resolvent converges in the Hilbert-Schmidt norm follows, i.e.
\[
\lim_{\eps\to0}
\big\|(-\Delta^\eps+\varkappa^2)^{-1}-
(-\Delta_\beta+\varkappa^2)^{-1}\big\|_2=0\,,
\]
and thus, a fortiori, the family $\{-\Delta^\eps\}_{\eps\geq0}$ approximates $-\Delta_\beta$ in the norm-resolvent topology.
\end{proof}

As an immediate corollary of Theorem~\ref{thm:ResolventConvergence} we get the following result.

\begin{thm} \label{thm:Approx_Op_Spec}
(i) The essential spectrum of $-\Delta^\eps$ is purely absolutely continuous and covers the nonnegative real axis, while the singularly continuous spectrum is empty.

(ii) If $\beta<0$, the operator $-\Delta^\eps$ has for all $\eps$ small enough exactly one negative eigenvalue $-\varkappa_\eps^2$ with the asymptotic behavior
\[
-\varkappa_\eps^2=
-\frac1{\beta^2B^2}+\OO(\eps)\,,\quad \eps\to0\,,
\]
which tends to the eigenvalue of the limit operator. If $\beta>0$, the perturbed operator has no eigenvalues.

(iii) If $\beta=0$, then there are two possibilities. If $\lambda_0<\frac1A$, the approximating operator has for all $\eps$ small enough exactly one negative eigenvalue $-\varkappa_\eps^2$ with the asymptotics
\[
-\varkappa_\eps^2
=
-\frac{(A-\frac1{\lambda}_0)^2}{\eps^2B^2}
+\OO\bigg(\frac1\eps\bigg)\,,\quad
\eps\to0\,,
\]
tending to $-\infty$. In the opposite case the operator $-\Delta^\eps$ has no eigenvalues.
\end{thm}
\begin{proof}
The arguments used in the proof of Theorem~\ref{thm:LimOpSpectr} together with the precise structure of the resolvent $(-\Delta^\eps-k^2)^{-1}$ give the first statement of the theorem. To obtain the remaining two claims we only need to observe that the resolvent $(-\Delta^\eps+\varkappa^2)^{-1}$ has only one pole at $\varkappa_\eps$ admitting the asymptotics
\[
\varkappa_\eps=\frac1B
\bigg(\frac{A-\frac1{\lambda_0}}\eps+
\frac{\lambda_1}{\lambda_0^2}+\OO(\eps)\bigg)\,,\quad
\eps\to0\,,
\]
and that the operator $\Delta^\eps$ can have only negative eigenvalues.
\end{proof}

\section{Convergence of the scattering matrices} \label{s:scattering}

In the final section we investigate stationary scattering for the pair $(-\Delta^\eps,-\Delta_0)$. Our aim is to show that the corresponding scattering amplitudes are close to those for the pair $(-\Delta_\beta,-\Delta_0)$ in the limit $\eps\to0$. We consider the incoming monochromatic wave $\exp(-\mathrm{i}kx_i)$ approaching the vertex along the edge $\gamma_i$. The corresponding scattering solution $\psi_i^\eps$ has to solve the problem
\begin{equation}\label{scattering:perturbed}
-\psi''+\frac{\lambda(\eps)}{\eps^3}V_\eps(x)
\langle\psi,V_\eps\rangle_\Gamma=k^2\psi
\quad\mathrm{on}\;\;\Gamma,\quad
\psi\in{K}(\Gamma)\,,
\end{equation}
and, by virtue of the compactness of the support of $V$, it takes the form
\[
\psi_i^\eps(x_j)=\delta_{ij}\exp(-\mathrm{i}kx_j)+
\mathcal{S}_{ij}^\eps(k)\exp(\mathrm{i}kx_j)\,,
\quad x_j\geq1\,.
\]
Hence to solve the scattering problem for the Hamiltonian $-\Delta^\eps$ we need to analyze behavior of the amplitudes $\mathcal{S}_{ij}^\eps(k)$ as the scaling parameter $\eps$ approaches zero.

The integro-differential equation \eqref{scattering:perturbed} for the scattering solution $\psi_i^\eps$ can easily be reformulated as an integral equation by using the variation-of-constants method,
\begin{multline}\label{eq:scat_sol}
\psi^\eps_i(x_j)=
-\frac{\lambda(\eps)\langle\psi^\eps_i,V_\eps
\rangle_\Gamma}{k\eps^3}
\int_{x_j}^\eps V_\eps(y_j)
\sin k(x_j-y_j)
\,\D y_j
\\
+\delta_{ij}\exp(-\mathrm{i}kx_j)+
\mathcal{S}^\eps_{ij}(k)\exp(\mathrm{i}kx_j)\,.
\end{multline}
Noting that
\begin{align*}
\psi_i\big|_{x_j=0}&=
\frac{\lambda(\eps)\langle\psi^\eps_i,V_\eps
\rangle_\Gamma}{k\eps^3}
\int_0^\eps V_\eps(y_j)\sin ky_j
\,\D y_j
+\delta_{ij}+
\mathcal{S}_{ij}^\eps(k)\,,
\\
\psi'_i\big|_{x_j=0}&=
-\frac{\lambda(\eps)\langle\psi^\eps_i,V_\eps
\rangle_\Gamma}{\eps^3}
\int_0^\eps V_\eps(y_j)\cos ky_j
\,\D y_j
-\mathrm{i}k\delta_{ij}+\mathrm{i}k
\mathcal{S}_{ij}^\eps(k)\,,
\end{align*}
we substitute these relations into the Kirchhoff matching conditions to conclude that
\begin{align}
\begin{aligned}\label{eq:scat_amplit}
\mathcal{S}^\eps_{ij}(k)&=
\frac{\lambda(\eps)\langle\psi^\eps_i,V_\eps
\rangle_\Gamma}{k\eps^3}
\Bigg[
-\int_0^\eps V_\eps(y_j)
\sin ky_j
\,\D y_j
\\
&
\phantom{=}\,
+\frac1{\mathrm{i}n}
\sum_{\ell=1}^n
\int_0^\eps V_\eps(y_\ell)\exp(\mathrm{i}ky_\ell)
\,\D y_\ell
\Bigg]
+\frac2n-\delta_{ij}
\\
&=
\frac{\lambda(\eps)\langle\psi^\eps_i,V_\eps
\rangle_\Gamma}{\eps}
\Bigg[
\sum_{\ell=1}^n
\bigg(\frac1n-\delta_{\ell j}\bigg)\vartheta_\ell
+\OO(\eps)
\Bigg]+\frac2n-\delta_{ij}.
\end{aligned}
\end{align}
Comparing formul{\ae} \eqref{eq:scat_sol} and \eqref{eq:scat_amplit}, we derive the following Fredholm integral equation (with a degenerate kernel) for the scattering solution,
\[
\psi_i(x_j)=
\langle\psi^\eps_i,V_\eps
\rangle_\Gamma W(x_j)+F(x_j),
\]
where
\begin{multline*}
W(x_j)=
\frac{\lambda(\eps)}{2\mathrm{i}k\eps^3}
\Bigg[
\int_{x_j}^\eps V_\eps(y_j)
\exp(\mathrm{i}k(y_j-x_j))
\,\D y_j
+
\int^{x_j}_0 V_\eps(y_j)
\exp(\mathrm{i}k(x_j-y_j))
\,\D y_j
\\
\phantom{=}\,
+
\sum_{\ell=1}^n
\bigg(\frac2n-\delta_{\ell j}\bigg)
\int_0^\eps V_\eps(y_\ell)
\exp(\mathrm{i}k(x_j+y_\ell))
\,\D y_\ell
\Bigg]
\end{multline*}
and
\[
F(x_j)=-2\mathrm{i}\delta_{ij}\sin kx_j
+\frac2n\exp(\mathrm{i}kx_j);
\]
then as an immediate consequence of this fact we get
\begin{equation}\label{v(eps)}
\langle\psi^\eps_i,V_\eps
\rangle_\Gamma=
\Bigg[
\sum_{j=1}^n\int_0^\eps F(x_j)V_\eps(x_j)
\,\D x_j
\Bigg]
\Bigg[
1-
\sum_{j=1}^n\int_0^\eps W(x_j)V_\eps(x_j)
\,\D x_j
\Bigg]^{-1}
.
\end{equation}
Next we are going to find the asymptotic behavior of the quantity $\langle\psi^\eps_i,V_\eps \rangle_\Gamma$ defined by \eqref{v(eps)}, which will be used further to get the asymptotics of the scattering amplitudes via \eqref{eq:scat_amplit}. To this end, we first denote the numerator of \eqref{v(eps)} as $N$ and analyze its asymptotic behavior,
\begin{align*}
N&=
-2\mathrm{i}
\int_0^\eps V_\eps(x_i)
\sin kx_i
\,\D x_i+
\frac2n\sum_{j=1}^n
\int_0^\eps V_\eps(x_j)
\exp(\mathrm{i}kx_j)
\,\D x_j
\\
&=
-2\mathrm{i}
\eps
\Bigg[
\int_0^1 V(x_i)
\sin k\eps x_i
\,\D x_i+
\frac{\mathrm{i}}n
\sum_{j=1}^n
\int_0^1 V(x_j)
\exp(\mathrm{i}k\eps x_j)
\,\D x_j\Bigg]
\\
&=
2\mathrm{i}k\eps^2
\sum_{j=1}^n\bigg(\frac1n-\delta_{ij}\bigg)\vartheta_j+\OO(\eps^3)\,,
\quad \eps\to0\,.
\end{align*}
Then the denominator of \eqref{v(eps)} can be written as $1-D$, where $D$ behaves as follows,
\begin{align*}
D&=
\frac{\lambda(\eps)}{2\mathrm{i}k\eps^3}
\Bigg[
\sum_{j=1}^n
\int_0^\eps\int_0^\eps
V_\eps(x_j)V_\eps(y_j)
\exp(\mathrm{i}k|x_j-y_j|)
\,\D x_j\D y_j
\\
&\phantom{=}\,
+
\sum_{j=1}^n\sum_{\ell=1}^n
\bigg(\frac2n-\delta_{\ell j}\bigg)
\int_0^\eps \int_0^\eps V_\eps(x_j)V_\eps(y_\ell)
\exp(\mathrm{i}k(x_j+y_\ell))
\,\D x_j\D y_\ell
\Bigg]
\\
&=
\frac{\lambda(\eps)}{2\mathrm{i}k\eps}
\Bigg[
\sum_{j=1}^n
\int_0^1\int_0^1
V(x_j)V(y_j)
\big(\exp(\mathrm{i}k\eps|x_j-y_j|)-
\exp(\mathrm{i}k\eps(x_j+y_j))\big)
\,\D x_j\D y_j
\\
&\phantom{=}\,
+
\frac2n
\sum_{j=1}^n\sum_{\ell=1}^n
\int_0^1 \int_0^1 V(x_j)V(y_\ell)
\exp(\mathrm{i}k\eps(x_j+y_\ell))
\,\D x_j\D y_\ell
\Bigg]
\\
&=\lambda(\eps)\big(A+\mathrm{i}k\eps B+\OO(\eps^2)\big)\,,
\quad\eps\to0\,,
\end{align*}
with the usual definitions of the constants $A$ and $B$. Combining the above asymptotic formul{\ae} for $N$ and $D$ along with \eqref{v(eps)}, we finally conclude that the scattering amplitude $\mathcal{S}_{ij}^\eps(k)$ has the following asymptotic behavior,
\begin{align*}
\mathcal{S}_{ij}^\eps&=\frac2n-\delta_{ij}
+\frac{2\mathrm{i}k\eps\lambda(\eps)}
{1-\lambda(\eps)( A+\mathrm{i}k\eps B)}
\\
&\phantom{=}\,\times
\Bigg[
\sum_{\ell=1}^n\bigg(\frac1n-\delta_{\ell i}\bigg)\vartheta_\ell
\Bigg]
\Bigg[
\sum_{\ell=1}^n\bigg(\frac1n-\delta_{\ell j}\bigg)\vartheta_\ell
\Bigg]
+\OO(\eps)
=\mathcal{S}_{ij}
+\OO(\eps)\,,
\quad \eps\to0\,,
\end{align*}
where the limit values $\mathcal{S}_{ij}$ are defined in Theorem~\ref{thm:LimOp_ScatMat}. In this way we have shown

\begin{thm}\label{thm:ScatMatConv}
For any momentum $k>0$ the on-shell scattering matrix for the pair $(-\Delta^\eps,-\Delta_0)$ converges as $\eps\to0$ to that of $(-\Delta_\beta,-\Delta_0)$, and moreover, there is a constant $C$ such that
\[
\|\mathcal{S}^\eps(k)-\mathcal{S}(k)\|
\leq C\eps\,,\quad \eps\in(0,1]\,,
\]
where $\|\cdot\|$ stands for the operator norm of the matrix.
\end{thm}

\subsection*{Acknowledgments}
The authors are grateful to Rostyslav Hryniv for careful reading of the manuscript and valuable remarks.
The second author thanks Yuri Golovaty for stimulating
discussions.
The research was supported by the Czech Science Foundation within
the project P203/11/0701 and by the European Union with the project ``Support for research teams on CTU'' CZ.1.07/2.3.00/30.0034.


\end{document}